\input ./style/arxiv-vmsta.cfg
\documentclass[numbers,compress,v1.0.1]{vmsta}

\volume{4}
\issue{1}
\pubyear{2017}
\firstpage{1}
\lastpage{13}
\doi{10.15559/16-VMSTA72}

\startlocaldefs

\newcommand{\rrvert}{\vert}
\newcommand{\llvert}{\vert}
\allowdisplaybreaks

\newtheorem{thm}{Theorem}
\newtheorem{lemma}{Lemma}
\newtheorem{cor}{Corollary}
\newtheorem{proposition}{Proposition}
\theoremstyle{definition}
\newtheorem{remark}{Remark}
\endlocaldefs

\begin{document}
\begin{frontmatter}

\title{Large deviations for i.i.d. replications of the total progeny
of a Galton--Watson process}

%
%
%

\author{\inits{C.}\fnm{Claudio}\snm{Macci}\corref{cor1}}\email{macci@mat.uniroma2.it}
\cortext[cor1]{Corresponding author.}
\address{Dipartimento di Matematica, Universit\`a di Roma Tor
Vergata,\\
Via della Ricerca Scientifica, I-00133 Rome, Italy}
\author{\inits{B.}\fnm{Barbara}\snm{Pacchiarotti}}\email{pacchiar@mat.uniroma2.it}

\markboth{C. Macci, B. Pacchiarotti}{Large deviations for i.i.d. replications of the total progeny
of a Galton--Watson process}

\begin{abstract}
The Galton--Watson process is the simplest example of a branching
process. The relationship between the offspring distribution, and,
when the extinction occurs almost surely, the distribution of the
total progeny is well known. In this paper, we illustrate the
relationship between these two distributions when we consider the
large deviation rate function (provided by Cram\'{e}r's theorem)
for empirical means of i.i.d.\ random variables. We also consider
the case with a random initial population. In the final part, we
present large deviation results for sequences of estimators of the
offspring mean based on i.i.d.\ replications of total progeny.
\end{abstract}


\begin{keywords}
\kwd{Cram\'{e}r's theorem}
\kwd{initial random population}
\kwd{estimators of offspring mean}
\end{keywords}
\begin{keywords}[2010]
\kwd{60F10}
\kwd{60J80}
\kwd{62F10}
\kwd{62F12}
\end{keywords}

\received{27 September 2016}
%
\revised{15 December 2016}
%
\accepted{17 December 2016}
\publishedonline{11 January 2017}
\end{frontmatter}

\section{Introduction}
There is a vast literature on branching processes. Here we cite
the monographs \cite{AsmussenHering,AthreyaNey,Harris}; moreover, we also cite the monographs \cite{Mode}
for the multitype case, \cite{Guttorp}, which focuses on
statistical inference, and \cite{Jagers} and \cite{KimmelAxelrod}
for applications in biology.

The simplest example of a branching process is the Galton--Watson
process. We consider the case of a population that has a unique
individual at the beginning and all the individuals (of all
generations) live for a unitary time; moreover, at the end of
their lifetimes, every individual of the population (of every
generation) produces a random number of new individuals acting
independently of all the rest, according to a specific fixed
distribution. So, if we consider a sequence of random variables
$\{V_n:n\geq0\}$ such that $V_n$ is the population size at time
$n$ (for all $n\geq0$),  we have $V_0=1$ and
\[
V_n:=\sum_{k=1}^{V_{n-1}}X_{n,k}
\quad (\mbox{for}\ n\geq1),
\]
where $\{X_{n,i}:n,i\geq1\}$ is a family of nonnegative integer-valued
i.i.d.\ random variables. In other words,
$X_{n,1},\ldots,X_{n,V_{n-1}}$ represent the offspring generated
at time $n$ by each of $V_{n-1}$ individuals that live at time
$n-1$. We recall some other preliminaries on the Galton--Watson process
in Section~\ref{sec:preliminaries}, where, in particular,
we consider a slightly different notation to allow the case
with a random initial population (instead of the case with a
unitary initial population cited before).

In this paper, we present large deviation results. The theory of
large deviations is a collection of techniques that gives an
asymptotic estimate of small probabilities in an exponential scale
(see, e.g., \cite{DemboZeitouni} as a reference). We recall some
preliminaries in Section~\ref{sec:preliminaries}.
The literature on large deviations for branching processes is
large. Here we essentially recall some references with results
concerning the Galton--Watson process.\looseness=-1

In several references, the large-time behavior for the
supercritical case is studied, namely the case where the
offspring mean $\mu$ is strictly larger than one (in such a case,
the extinction probability is strictly less than one). Here we
recall \cite{Athreya} (see also \cite{AthreyaVidyashankar} for
the multitype case), \cite{BigginsBingham}, where the main object
is the study of the tails of $W:=\lim_{n\to\infty}V_n/\mu^n$,
\cite{NeyVidyashankar2003} with a careful analysis based on
harmonic moments of $\{V_n:n\geq0\}$, \cite{NeyVidyashankar2004}
(and \cite{NeyVidyashankar2006}) with some conditional large
deviation results based on some local limit theorems,
\cite{FleischmannWachtel} where the central role of some \lq\lq
lower deviation probabilities\rq\rq\ is highlighted for the study
of the asymptotic behavior of the Lotka--Nagaev estimator
$V_{n+1}/V_n$ of $\mu$.

Other references study the most likely paths to extinction at some
time $n_0$ when the initial population $k$ is large. The idea is
to consider the representation of a branching process with initial
population equal to $k$ as a sum of $k$ i.i.d.\ replications of the
process with a unitary initial population; in this case,
Cram\'{e}r's theorem for empirical means of i.i.d.\ random
variables (on $R^{n_0}$) plays a crucial role. A most likely path
to extinction in \cite{KlebanerLiptser2006} (see also
\cite{KlebanerLiptser2008}) is a trajectory that minimizes the
rate function among the paths that reach the level 0 at time
$n_0$. A generalization of this concept for the most likely paths
to reach a level $b\geq0$ can be found in \cite{HamzaKlebaner}.

In this paper, we are interested in a different direction. Namely,
we are interested in the empirical means of i.i.d.\ replications of
the total progeny of a~Galton--Watson process. The total progenies
of branching processes are studied in several references: here we
cite the old references \cite{Dwass,Kennedy,Pakes} for a~Galton--Watson process,
and \cite{GonzalezMolina}
(see Section~2.2) among the references concerning different
branching processes. The total progeny of a Galton--Watson process
is an almost surely finite random variable when the extinction
occurs almost surely, and therefore the supercritical case will
not be considered. Some relationships between the offspring
distribution and the total progeny distribution of a Galton--Watson
process are well known (see \eqref{eq:link-pmf} for the
probability mass functions and \eqref{eq:link-pgf} for the
probability generating functions).

A new relationship is provided by Proposition
\ref{prop:main-unitary-initial-population}, where we illustrate
how the rate function for the empirical means of total progenies
can be expressed in terms of the analogous rate function for the
empirical means of a single progeny. This is a quite natural
problem to investigate large deviations, and, as we can expect,
\eqref{eq:link-pgf} has an important role in the proof; in fact,
the large deviation rate function for empirical means of i.i.d.\
random variables (provided by Cram\'{e}r's theorem recalled below;
see Theorem \ref{th:Cramer}) is given by the Legendre transform of
the logarithm of the (common) moment generating function of the
random variables. Moreover, the relationship provided by
Proposition \ref{prop:main-unitary-initial-population} can have
interest in information theory because the involved rate functions
can be expressed in terms of suitable relative entropies (or
Kullback--Leibler divergences); see, for example, \cite{Varadhan} for a
discussion on the rate function expressions in terms of the
relative entropy.

Another result presented in this paper is Proposition
\ref{prop:main}, that is a version of Proposition~\ref{prop:main-unitary-initial-population}, where the initial
population $V_0$ is a random variable with a suitable
distribution. Finally, in Propositions \ref{prop:main-estimators}
and \ref{prop:minor-estimators}, we prove large deviation results
for some estimators of the offspring mean $\mu$ in terms of i.i.d.\
replications of the total progeny and of the initial population
(we are considering the case where the initial population $V_0$ is
a random variable as in Proposition \ref{prop:main}).

We conclude with the outline of the paper. We start with some
preliminaries in Section~\ref{sec:preliminaries}. In Section~\ref{sec:applications-CT}, we prove the results concerning the
large deviation rate functions related to Cram\'{e}r's theorem.
Finally, in Section~\ref{sec:estimators}, we prove the large
deviation results for the estimators of the offspring mean $\mu$.

\section{Preliminaries}\label{sec:preliminaries}
We start with some preliminaries on the Galton--Watson process. In the
second part, we recall some preliminaries on large deviations.

\subsection{Preliminaries on Galton--Watson process}
Here we introduce a slightly different notation, and, moreover, we
recall some preliminaries in order to define the total progeny of
a Galton--Watson process.

We start with some notation concerning the offspring distribution
(note that $\mu_f$ defined further coincides with $\mu$ in the
Introduction):
\begin{itemize}
\item the probability mass function $p_h:=P(X_{n,i}=h)$ (for all
integer $h\geq0$);
\item the probability generating function $f(s):=\sum_{h\geq0}s^hp_h$;
\item the mean value $\mu_f:=\sum_{h\geq0}hp_h$ (and we have $\mu
_f=f^\prime(1)$).
\end{itemize}
Moreover, we introduce the analogous items for the initial
population:
\begin{itemize}
\item the probability mass function $\{q_r:r\geq0\}$ (see \eqref
{eq:pmf-initial-population});
\item the probability generating function $g(s):=\sum_{r\geq0}s^rq_r$;
\item the mean value $\mu_g:=\sum_{r\geq0}rq_r$ (and we have $\mu
_g=g^\prime(1)$).
\end{itemize}
So, from now on, we consider the following slightly different
notation:
\[
\bigl\{V_n^{f,g}:n\geq0 \bigr\}\vadjust{\eject}
\]
(in place of $\{V_n:n\geq0\}$ presented before). More precisely:
\begin{itemize}
\item the probability generating function of $V_0^{f,g}$ is $g$ (so
$V_0^{f,g}$ does not depend on $f$), and therefore
\begin{equation}
\label{eq:pmf-initial-population} q_r:=P \bigl(V_0^{f,g}=r \bigr)\quad
(\mbox{for all integer}\ r\geq0);
\end{equation}
\item for a family of i.i.d.\ random variables $\{X_{n,i}:n,i\geq
1\}$ with probability generating function $f$, we have
\[
V_n^{f,g}:=\sum_{i=1}^{V_{n-1}^{f,g}}X_{n,i}
\quad (\mbox{for all}\ n\geq1).
\]
\end{itemize}

\begin{remark}\label{rem:unitary-initial-population}
Note that $ \{V_n^{f,g}:n\geq0 \}$ here corresponds to
$\{V_n:n\geq0\}$ presented in the Introduction if $q_1=1$ or,
equivalently, if $g=\mathrm{id}$ \textup{(}i.e. $g(s)=s$ for all
$s$\textup{)}.
\end{remark}

If we consider the extinction probability
\[
p_{\mathrm{ext}}^{f,g}:=P \bigl( \bigl\{V_n^{f,g}=0
\ \mbox{for some}\ n\geq0 \bigr\} \bigr),
\]
then it is known that we have
\[
p_{\mathrm{ext}}^{f,\mathrm{id}}=\min\bigl\{s\in[0,1]:f(s)=s\bigr\};
\]
moreover, if $p_0>0$, then we have
$p_{\mathrm{ext}}^{f,\mathrm{id}}=1$ if $\mu_f\leq1$ and
$p_{\mathrm{ext}}^{f,\mathrm{id}}\in(0,1)$ if $\mu_f>1$. More
generally, we have
\[
p_{\mathrm{ext}}^{f,g}:=q_0+\sum
_{n\geq1}\bigl(p_{\mathrm{ext}}^{f,\mathrm
{id}}
\bigr)^nq_n=g\bigl(p_{\mathrm{ext}}^{f,\mathrm{id}}\bigr),
\]
and, if $q_0<1$ (we obviously have $p_{\mathrm{ext}}^{f,g}=1$ if
$q_0=1$), then we have the following cases:
\[
\begin{array}{ll}
p_{\mathrm{ext}}^{f,g}=g(0)=q_0&\ \mbox{if}\ p_0=0;\\
p_{\mathrm{ext}}^{f,g}=g(1)=1&\ \mbox{if}\ p_0>0\ \mbox{and}\ \mu_f\leq
1;\\
p_{\mathrm{ext}}^{f,g}\in(q_0,1)&\ \mbox{if}\ p_0>0\ \mbox{and}\
\mu_f>1.
\end{array} %
\]
Then, if $p_0>0$ and $\mu_f\leq1$, then the random variable $Y^{f,g}$
defined by
\[
Y^{f,g}:=\sum_{i=0}^{\tau-1}V_i^{f,g},
\quad \mbox{where}\ \tau:=\inf \bigl\{ n\geq0:V_n^{f,g}=0 \bigr
\},
\]
is almost surely finite and provides the total progeny of
$ \{V_n^{f,g}:n\geq0 \}$. In view of what follows, we
consider the probability generating function
\[
\mathcal{G}_{f,g}(s):=\sum_{k\geq0}s^k
\pi_k^{f,g},
\]
where $ \{\pi_k^{f,g}:k\geq0 \}$ is the probability mass
function of the random variable $Y^{f,g}$. Moreover, we have the
mean value
\begin{equation}
\label{eq:mean-value-total-progeny} \nu^{f,g}:=\sum_{k\geq0}k
\pi_k^{f,g},\quad \mbox{and we have}\quad \nu^{f,g}=
\frac{\mu_g}{1-\mu_f};\vadjust{\eject}
\end{equation}
in particular, $\nu^{f,g}=\frac{\mu_g}{1-\mu_f}$ even if
$\mu_f=1$, namely
\[
\nu^{f,g}= \left\{ %
\begin{array}{@{}ll}
\infty&\ \mbox{if}\ \mu_g>0\ (\mbox{and}\ \mu_f=1)\\
0&\ \mbox{if}\ \mu_g=0\ (\mbox{and}\ \mu_f=1).
\end{array} %
\right.
\]

Finally, we recall some well-known connections between total
progeny and offspring distributions (see e.g. \cite{Dwass}): for
the probability mass functions, we have
\begin{equation}
\label{eq:link-pmf} \pi_k^{f,\mathrm{id}}=\frac{1}{k}\cdot
p_{k-1}^{*k},
\end{equation}
where $\{p_h^{*n}:h\geq0\}$ is the $n$th power of convolution of
$\{p_h:h\geq0\}$; for the probability generating functions, we
have
\begin{equation}
\label{eq:link-pgf} \mathcal{G}_{f,\mathrm{id}}(s)=sf\bigl(\mathcal{G}_{f,\mathrm{id}}(s)
\bigr).
\end{equation}

\subsection{Preliminaries on large deviations}
We start with the concept of large deviation principle (LDP). A
sequence of random variables $\{W_n:n\geq1\}$ taking
values in a topological space $\mathcal{W}$ satisfies the LDP with
rate function $I:\mathcal{W}\to[0,\infty]$ if $I$ is a lower
semicontinuous function,
\[
\liminf_{n\to\infty}\frac{1}{n}\log P(W_n\in O)
\geq-\inf_{w\in O}I(w)\quad \mbox{for all open sets}\ O,
\]
and
\[
\limsup_{n\to\infty}\frac{1}{n}\log P(W_n\in C)
\leq-\inf_{w\in C}I(w)\quad \mbox{for all closed sets}\ C.
\]
We also recall that a rate function $I$ is said to be good if all
its level sets $\{\{w\in\mathcal{W}:I(w)\leq\eta\}:\eta\geq0\}$
are compact.

\begin{remark}\label{rem:closed-sets-with-probability-1}
If $P(W_n\in S)=1$ for some closed set $S$ (at least
eventually with respect to $n$), then $I(w)=\infty$ for
$w\notin S$; this can be checked by taking the lower bound for
the open set $O=S^c$.
\end{remark}

In particular, we refer to Cram\'{e}r's theorem on $\mathbb{R}^d$
(see e.g. Theorems 2.2.3\break and 2.2.30 in \cite{DemboZeitouni} for
the cases $d=1$ and $d\geq2$), and we recall its statement. We
remark that, in this paper, we consider the cases $d=1$ (in such a
case, the rate function need not to be a good rate function) and
$d=2$. Moreover, we use the symbol
$\langle\cdot,\cdot\rangle$ for the inner product in
$\mathbb{R}^d$.

\begin{thm}[Cram\'{e}r's theorem]\label{th:Cramer}
Let $\{W_n:n\geq1\}$ be a sequence of i.i.d.\
$\mathbb{R}^d$-valued random variables, and let $\{\bar{W}_n:n\geq1\}$
be the
sequence of empirical means defined by
$\bar{W}_n:=\frac{1}{n}\sum_{k=1}^nW_k$ \textup{(}for all $n\geq1)$.

\textup{(i)} If $d=1$, then $\{\bar{W}_n:n\geq1\}$ satisfies the LDP with
rate function $I$ defined by
\[
I(w):=\sup_{\theta\in\mathbb{R}} \bigl\{\theta w-\log\mathbb{E}
\bigl[e^{\theta W_1} \bigr] \bigr\}.
\]

\textup{(ii)} If $d\geq2$ and the origin of $\mathbb{R}^d$ belongs to the
interior of the set
$\{\theta\in\mathbb{R}^d:\log\mathbb{E} [e^{\langle\theta,W_1\rangle
} ]<\infty\}$,
then $\{\bar{W}_n:n\geq1\}$ satisfies the LDP with good rate
function $I$ defined by
\[
I(w):=\sup_{\theta\in\mathbb{R}^d} \bigl\{\langle\theta,w\rangle-\log \mathbb{E}
\bigl[e^{\langle\theta,W_1\rangle} \bigr] \bigr\}.
\]
\end{thm}

\section{Applications of Cram\'{e}r's theorem}\label{sec:applications-CT}
The aim of this section is to prove Propositions
\ref{prop:main-unitary-initial-population} and \ref{prop:main}. In
view of this, we recall Lemmas \ref{lem:LDP-offspring} and
\ref{lem:LDP-totalprogeny}, which give two immediate applications
of Cram\'{e}r's theorem (Theorem \ref{th:Cramer}) with $d=1$; in
Lemma \ref{lem:LDP-totalprogeny}, we consider the case with a
unitary initial population almost surely (thus, as stated Remark
\ref{rem:unitary-initial-population}, the case with $q_1=1$ or,
equivalently, $g=\mathrm{id}$).

\begin{lemma}[Cram\'{e}r's theorem for offspring distribution]\label
{lem:LDP-offspring}
Let $\{X_n:n\geq1\}$ be i.i.d.\ random variables with probability
generating function $f$. Let $\{\bar{X}_n:n\geq1\}$ be the
sequence of empirical means defined by
$\bar{X}_n:=\frac{1}{n}\sum_{k=1}^nX_k$ \textup{(}for all $n\geq1)$. Then
$\{\bar{X}_n:n\geq1\}$ satisfies the LDP with rate function $I_f$
defined by $I_f(x):=\sup_{\alpha\in\mathbb{R}}\{\alpha x-\log
f(e^\alpha)\}$.
\end{lemma}

\begin{lemma}[Cram\'{e}r's theorem for total progeny distribution with
$g=\mathrm{id}$]\label{lem:LDP-totalprogeny}
Assume that $p_0>0$ and $\mu_f\leq1$. Let $\{Y_n:n\geq1\}$ be
i.i.d.\ random variables with probability generating function
$\mathcal{G}_{f,\mathrm{id}}$. Let $\{\bar{Y}_n:n\geq1\}$ be the
sequence of empirical means defined by
$\bar{Y}_n:=\frac{1}{n}\sum_{k=1}^nY_k$ \textup{(}for all $n\geq1)$. Then
$\{\bar{Y}_n:n\geq1\}$ satisfies the LDP with rate function
$I_{\mathcal{G}_{f,\mathrm{id}}}$ defined by
$I_{\mathcal{G}_{f,\mathrm{id}}}(y):=\sup_{\beta\in\mathbb{R}}\{\beta
y-\log\mathcal{G}_{f,\mathrm{id}}(e^\beta)\}$.
\end{lemma}

Now we can prove our main results. We start with Proposition
\ref{prop:main-unitary-initial-population}, which provides an
expression for $I_{\mathcal{G}_{f,\mathrm{id}}}$ in terms of
$I_f$.

\begin{proposition}\label{prop:main-unitary-initial-population}
Let $I_f$ and $I_{\mathcal{G}_{f,\mathrm{id}}}$ be the rate
functions in Lemmas \ref{lem:LDP-offspring} and
\ref{lem:LDP-totalprogeny}. Then we have
$I_{\mathcal{G}_{f,\mathrm{id}}}(y)=yI_f (\frac{y-1}{y} )$
for all $y\geq1$.
\end{proposition}
\begin{proof}
We remark that
\[
I_f(x):=\sup_{\alpha\in\mathcal{D}(f)}\bigl\{\alpha x-\log f
\bigl(e^\alpha\bigr)\bigr\},
\]
where
$\mathcal{D}(f):=\{\alpha\in\mathbb{R}:f(e^\alpha)<\infty\}$, and
\[
I_{\mathcal{G}_{f,\mathrm{id}}}(x):=\sup_{\beta\in\mathcal{D}(\mathcal
{G}_{f,\mathrm{id}})}\bigl\{\beta y-\log
\mathcal{G}_{f,\mathrm{id}}\bigl(e^\beta \bigr)\bigr\},
\]
where
$\mathcal{D}(\mathcal{G}_{f,\mathrm{id}}):=\{\beta\in\mathbb{R}:\mathcal
{G}_{f,\mathrm{id}}(e^\beta)<\infty\}$,
by Lemmas \ref{lem:LDP-offspring} and \ref{lem:LDP-totalprogeny},
respectively.

Moreover, the function
$\alpha:\mathcal{D}(\mathcal{G}_{f,\mathrm{id}})\to\mathcal{D}(f)$
defined by
\[
\alpha(\beta):=\log\mathcal{G}_{f,\mathrm{id}}\bigl(e^\beta\bigr)
\]
is a bijection. This can be checked noting that
$\alpha(\beta)\in\mathcal{D}(f)$ (for all
$\beta\in\mathcal{D}(\mathcal{G}_{f,\mathrm{id}})$) because
$f(e^{\alpha(\beta)})=f(\mathcal{G}_{f,\mathrm{id}}(e^\beta))=\frac
{\mathcal{G}_{f,\mathrm{id}}(e^\beta)}{e^\beta}<\infty$
(here we take into account \eqref{eq:link-pgf}); moreover, its
inverse
$\beta:\mathcal{D}(f)\to\mathcal{D}(\mathcal{G}_{f,\mathrm{id}})$
is defined by
\[
\beta(\alpha):=\log\mathcal{G}_{f,\mathrm{id}}^{-1}
\bigl(e^\alpha\bigr)
\]
(where $\mathcal{G}_{f,\mathrm{id}}^{-1}$ is the inverse of
$\mathcal{G}_{f,\mathrm{id}}$), and
$\beta(\alpha)\in\mathcal{D}(\mathcal{G}_{f,\mathrm{id}})$ (for
all $\alpha\in\mathcal{D}(f)$) because
$\mathcal{G}_{f,\mathrm{id}}(e^{\beta(\alpha)})=e^\alpha<\infty$.

Thus, we can set $\alpha=\log\mathcal{G}_{f,\mathrm{id}}(e^\beta)$
(for $\beta\in\mathcal{D}(\mathcal{G}_{f,\mathrm{id}})$) in the
expression of $I_f(x)$, and we get
\[
I_f(x)=\sup_{\beta\in\mathcal{D}(\mathcal{G}_{f,\mathrm{id}})}\bigl\{\log
\mathcal{G}_{f,\mathrm{id}}\bigl(e^\beta\bigr)x-\log f\bigl(
\mathcal{G}_{f,\mathrm
{id}}\bigl(e^\beta\bigr)\bigr)\bigr\}.
\]
Then (we take into account \eqref{eq:link-pgf} in the second
equality below)
\begin{align*}
I_f(x)&=\sup_{\beta\in\mathcal{D}(\mathcal{G}_{f,\mathrm{id}})}\bigl\{\log
\mathcal{G}_{f,\mathrm{id}}\bigl(e^\beta\bigr)x-\log(e^{-\beta}e^\beta
f\bigl(\mathcal{G}_{f,\mathrm{id}}\bigl(e^\beta\bigr)\bigr)\bigr\}
\\
&=\sup_{\beta\in\mathcal{D}(\mathcal{G}_{f,\mathrm{id}})}\bigl\{\log\mathcal {G}_{f,\mathrm{id}}
\bigl(e^\beta\bigr)x+\beta-\log\mathcal{G}_{f,\mathrm
{id}}
\bigl(e^\beta\bigr)\bigr\}
\\
&=\sup_{\beta\in\mathcal{D}(\mathcal{G}_{f,\mathrm{id}})}\bigl\{\beta -(1-x)\log\mathcal{G}_{f,\mathrm{id}}
\bigl(e^\beta\bigr)\bigr\},
\end{align*}
and, for $x\in[0,1)$, we get
\[
I_f(x)=(1-x)I_{\mathcal{G}_{f,\mathrm{id}}} \biggl(\frac{1}{1-x} \biggr).
\]
We conclude by taking $x=\frac{y-1}{y}$ for $y\geq1$ (thus,
$x\in[0,1)$), and we obtain the desired equality with some easy
computations.
\end{proof}

Now we present Proposition \ref{prop:main}, which concerns the LDP
for the empirical means of i.i.d. bivariate random variables
$\{(Y_n,Z_n):n\geq1\}$ distributed as $(Y^{f,g},V_0^{f,g})$. In
particular, we obtain an expression for the rate function
$I_{\mathcal{G}_{f,g},g}$ in terms of $I_f$ in Lemma
\ref{lem:LDP-offspring} and $I_g$ defined by
\begin{equation}
\label{def:rf-initial-population} I_g(z):=\sup_{\gamma\in\mathbb{R}}\bigl\{\gamma z-
\log g\bigl(e^\gamma\bigr)\bigr\}.
\end{equation}

\begin{proposition}\label{prop:main}
Let $\{(Y_n,Z_n):n\geq1\}$ be i.i.d.\ random variables distributed
as $(Y^{f,g},V_0^{f,g})$. Assume that $\mathbb{E} [e^{\beta
Y^{f,g}+\gamma V_0^{f,g}} ]$ is finite in a neighborhood of
$(\beta,\gamma)=(0,0)$. Let $\{(\bar{Y}_n,\bar{Z}_n):n\geq1\}$ be
the sequence of empirical means defined by\break
$(\bar{Y}_n,\bar{Z}_n):= (\frac{1}{n}\sum_{k=1}^nY_k,\frac{1}{n}\sum_{k=1}^nZ_k )$
\textup{(}for all $n\geq1)$. Then $\{(\bar{Y}_n,\bar{Z}_n):n\geq1\}$
satisfies the LDP with good rate function
$I_{\mathcal{G}_{f,g},g}$ defined by
\[
I_{\mathcal{G}_{f,g},g}(y,z)= \left\{ %
\begin{array}{@{}ll}
yI_f (\frac{y-z}{y} )+I_g(z)&\ \mbox{if}\ y\geq z>0,\\
I_g(0)&\ \mbox{if}\ y=z=0,\\
\infty&\ \mbox{otherwise}.
\end{array} %
\right.
\]
\end{proposition}

\begin{remark}\label{rem:implicit-hypotheses-for-prop-main}
We are assuming \textup{(}implicitly\textup{)} that $p_0>0$ and $\mu
_f\leq1$;
in fact, since we require that $\mathbb{E} [e^{\beta
Y^{f,g}+\gamma V_0^{f,g}} ]$ is finite in a neighborhood of
$(\beta,\gamma)=(0,0)$, we are assuming that $\mu_f<1$ and
$\mu_g<\infty$.
\end{remark}

\begin{proof}
The LDP is a consequence of Cram\'{e}r's theorem (Theorem
\ref{th:Cramer}) with $d=2$, and the rate function
$I_{\mathcal{G}_{f,g},g}$ is defined by
\[
I_{\mathcal{G}_{f,g},g}(y,z):=\sup_{\beta,\gamma\in\mathbb{R}} \bigl\{ \beta y+\gamma z-
\log\mathbb{E} \bigl[e^{\beta Y^{f,g}+\gamma
V_0^{f,g}} \bigr] \bigr\}.
\]
Throughout the proof, we restrict our attention on the pairs
$(y,z)$ such that $y\geq z\geq0$. In fact, almost surely, we have
$Y^{f,g}\geq V_0^{f,g}\geq0$, and therefore
$\bar{Y}_n\geq\bar{Z}_n\geq0$; thus, by Remark
\ref{rem:closed-sets-with-probability-1} we have
$I_{\mathcal{G}_{f,g},g}(y,z)=\infty$ if condition $y\geq z\geq0$
fails.

We remark that
$\mathbb{E} [s^{Y^{f,g}}|V_0^{f,g} ]=(\mathcal{G}_{f,\mathrm
{id}}(s))^{V_0^{f,g}}$,
and therefore
\[
\mathbb{E} \bigl[e^{\beta Y^{f,g}+\gamma V_0^{f,g}} \bigr]=\mathbb {E} \bigl[e^{\gamma V_0^{f,g}}
\bigl(\mathcal{G}_{f,\mathrm{id}}\bigl(e^\beta \bigr)\bigr)^{V_0^{f,g}}
\bigr] =g\bigl(e^\gamma\mathcal{G}_{f,\mathrm{id}}\bigl(e^\beta
\bigr)\bigr);
\]
thus,
\[
I_{\mathcal{G}_{f,g},g}(y,z)=\sup_{\beta,\gamma\in\mathbb{R}} \bigl\{ \beta y+\gamma z-\log
g\bigl(e^{\gamma+\log\mathcal{G}_{f,\mathrm
{id}}(e^\beta)}\bigr) \bigr\}.
\]
Furthermore, the function
\[
(\beta,\gamma)\mapsto\bigl(\beta,\gamma+\log\mathcal{G}_{f,\mathrm
{id}}
\bigl(e^\beta\bigr)\bigr)
\]
is a bijection defined on
$\mathcal{D}(\mathcal{G}_{f,\mathrm{id}})\times\mathbb{R}$, where
\[
\mathcal{D}(\mathcal{G}_{f,\mathrm{id}}):=\bigl\{\beta\in\mathbb{R}:\mathcal
{G}_{f,\mathrm{id}}\bigl(e^\beta\bigr)<\infty\bigr\}
\]
as in the proof of Proposition
\ref{prop:main-unitary-initial-population}; then, for
$\delta:=\gamma+\log\mathcal{G}_{f,\mathrm{id}}(e^\beta)$, we
obtain
\[
I_{\mathcal{G}_{f,g},g}(y,z)=\sup_{\beta,\delta\in\mathbb{R}} \bigl\{ \beta y+\bigl(\delta-
\log\mathcal{G}_{f,\mathrm{id}}\bigl(e^\beta\bigr)\bigr)z-\log g
\bigl(e^\delta\bigr) \bigr\}.
\]
Thus, we have (note that the last equality holds by Proposition
\ref{prop:main-unitary-initial-population})
\begin{align*}
I_{\mathcal{G}_{f,g},g}(y,z)&\leq\sup_{\beta\in\mathbb{R}} \bigl\{\beta y+z\log
\mathcal{G}_{f,\mathrm{id}}\bigl(e^\beta\bigr) \bigr\}+ \sup
_{\delta\in\mathbb{R}} \bigl\{\delta z-\log g\bigl(e^\delta\bigr) \bigr
\}
\\
&= \left\{ %
\begin{array}{@{}ll}
zI_{\mathcal{G}_{f,\mathrm{id}}}(y/z)+I_g(z)&\ \mbox{if}\ y\geq z>0,\\
I_g(0)&\ \mbox{if}\ y=z=0,\\
\infty&\ \mbox{otherwise.}
\end{array} %
\right.
\\
&= \left\{ %
\begin{array}{@{}ll}
yI_f (\frac{y-z}{y} )+I_g(z)&\ \mbox{if}\ y\geq z>0,\\
I_g(0)&\ \mbox{if}\ y=z=0,\\
\infty&\ \mbox{otherwise}.
\end{array} %
\right.
\end{align*}
We conclude by showing the inverse inequality
\begin{equation}
\label{eq:inverse-inequality} I_{\mathcal{G}_{f,g},g}(y,z)\geq\sup_{\beta\in\mathbb{R}} \bigl\{\beta
y+z\log\mathcal{G}_{f,\mathrm{id}}\bigl(e^\beta\bigr) \bigr\}+\sup
_{\delta\in
\mathbb{R}} \bigl\{\delta z-\log g\bigl(e^\delta\bigr) \bigr
\}.
\end{equation}
To this end, we take two sequences $\{\beta_n:n\geq1\}$
and $\{\delta_n:n\geq1\}$ such that
\[
\lim_{n\to\infty}\beta_ny-z\log\mathcal{G}_{f,\mathrm{id}}
\bigl(e^{\beta_n}\bigr) =\sup_{\beta\in\mathbb{R}} \bigl\{\beta y+z\log
\mathcal{G}_{f,\mathrm{id}}\bigl(e^\beta\bigr) \bigr\}
\]
and
\[
\lim_{n\to\infty}\delta_n z-\log g\bigl(e^{\delta_n}
\bigr)=\sup_{\delta\in\mathbb
{R}} \bigl\{\delta z-\log g\bigl(e^\delta
\bigr) \bigr\}.
\]
Then we have
\[
I_{\mathcal{G}_{f,g},g}(y,z)\geq\beta_n y+\bigl(\delta_n-\log
\mathcal {G}_{f,\mathrm{id}}\bigl(e^{\beta_n}\bigr)\bigr)z-\log g
\bigl(e^{\delta_n}\bigr),
\]
and we get \eqref{eq:inverse-inequality} letting $n$ go to
infinity.
\end{proof}

\section{Large deviations for estimators of $\mu_f$}\label{sec:estimators}
In this section, we prove two LDPs for two sequences of estimators
of the offspring mean $\mu_f$. Namely, if
$\{(\bar{Y}_n,\bar{Z}_n):n\geq1\}$ is the sequence in Proposition~\ref
{prop:main} (see also the precise assumptions in Remark
\ref{rem:implicit-hypotheses-for-prop-main}; in particular, we have
$\mu_f<1$), then we consider:
\begin{enumerate}
\item$ \{\frac{\bar{Y}_n-\bar{Z}_n}{\bar{Y}_n}:n\geq1 \}$;
\item$ \{\frac{\bar{Y}_n-\mu_g}{\bar{Y}_n}:n\geq1 \}$.
\end{enumerate}
Obviously, these estimators are well defined if the denominators
$\bar{Y}_n$ are different from zero; then, in order to have
well-defined estimators, we always assume that $q_0=0$ (where
$q_0$ is as in \eqref{eq:pmf-initial-population}), and, noting that,
in general, $I_g(0)=-\log q_0$, we have
\[
I_g(0)=\infty.
\]
Moreover, both sequences converge to
$\frac{\nu^{f,g}-\mu_g}{\nu^{f,g}}=\mu_f$ as $n\to\infty$ (see
$\nu^{f,g}$ in~\eqref{eq:mean-value-total-progeny}), and they coincide
when the initial population is deterministic
(equal to $\mu_g$ almost surely).

The LDPs of these two sequences are proved in Propositions
\ref{prop:main-estimators} and \ref{prop:minor-estimators}.
Moreover, Corollary \ref{cor:comparison} and Remark
\ref{rem:comparison} concern the comparison between the
convergence of the first sequence
$ \{\frac{\bar{Y}_n-\bar{Z}_n}{\bar{Y}_n}:n\geq1 \}$
and its analogue when the initial population is deterministic
(equal to the mean). Propositions \ref{prop:main-estimators} and
\ref{prop:minor-estimators} are proved by combining the
contraction principle (see e.g. Theorem 4.2.1 in
\cite{DemboZeitouni}) and Proposition \ref{prop:main} (note that
the rate function $I_{\mathcal{G}_{f,g},g}$ in Proposition
\ref{prop:main} is good, as it is required to apply the
contraction principle). We remark that, in the proofs of Propositions
\ref{prop:main-estimators} and
\ref{prop:minor-estimators}, we take into account that
$I_{\mathcal{G}_{f,g},g}(0,0)=\infty$ by Proposition
\ref{prop:main} and $I_g(0)=\infty$. At the end of this section, we
present some remarks on the comparison between the rate functions
in Propositions \ref{prop:main-estimators} and~\ref
{prop:minor-estimators} (Remarks
\ref{rem:rf-in-propminorestimator-could-be-finite-for-negative-arguments}
and \ref{rem:no-offsprings}).

We start with the LDP of the first sequence of estimators.

\begin{proposition}\label{prop:main-estimators}
Assume the same hypotheses of Proposition \ref{prop:main} and
$q_0=0$. Let $\{(Y_n,Z_n):n\geq1\}$ be i.i.d. random variables
distributed as $(Y^{f,g},V_0^{f,g})$. Let\break
$\{(\bar{Y}_n,\bar{Z}_n):n\geq1\}$ be the sequence of empirical
means defined by
$(\bar{Y}_n,\bar{Z}_n):= (\frac{1}{n}\sum_{k=1}^nY_k,\frac{1}{n}\sum_{k=1}^nZ_k )$
\textup{(}for all $n\geq1$\textup{)}. Then
$ \{\frac{\bar{Y}_n-\bar{Z}_n}{\bar{Y}_n}:n\geq1 \}$
satisfies the LDP with good rate function
$J_{\mathcal{G}_{f,g},g}$ defined by
\[
J_{\mathcal{G}_{f,g},g}(x):= \left\{ %
\begin{array}{@{}ll}
-\log g (e^{-\frac{I_f(x)}{1-x}} )&\ \mbox{if}\ x\in[0,1),\\
\infty&\ \mbox{otherwise}.
\end{array} %
\right.
\]
\end{proposition}
\begin{proof}
By Proposition \ref{prop:main} and the contraction principle we
have the LDP of
$ \{\frac{\bar{Y}_n-\bar{Z}_n}{\bar{Y}_n}:n\geq1 \}$
with good rate function $J_{\mathcal{G}_{f,g},g}$ defined by
\[
J_{\mathcal{G}_{f,g},g}(x):=\inf \biggl\{I_{\mathcal
{G}_{f,g},g}(y,z):y\geq z>0,
\frac{y-z}{y}=x \biggr\}.
\]
The case $x\notin[0,1)$ is trivial because we have the infimum
over the empty set. For $x\in[0,1)$, we rewrite this
expression as follows (where we take into account the expression
of the rate function $I_{\mathcal{G}_{f,g},g}$ in Proposition
\ref{prop:main}):\vadjust{\eject}
\begin{align*}
J_{\mathcal{G}_{f,g},g}(x)&=\inf \biggl\{I_{\mathcal{G}_{f,g},g} \biggl(\frac
{z}{1-x},z
\biggr):z>0 \biggr\}
\\
&=\inf \biggl\{\frac{z}{1-x}I_f \biggl(\frac{\frac{z}{1-x}-z}{\frac
{z}{1-x}}
\biggr)+I_g(z):z>0 \biggr\}
\\
&=\inf \biggl\{\frac{z}{1-x}I_f(x)+I_g(z):z>0
\biggr\}\\
&=-\sup \biggl\{-z\frac
{I_f(x)}{1-x}-I_g(z):z>0 \biggr\};
\end{align*}
thus, since $I_g(z)=\infty$ for $z\leq0$, we obtain
$J_{\mathcal{G}_{f,g},g}(x)=-\log
g (e^{-\frac{I_f(x)}{1-x}} )$ by taking into account the
definition of $I_g$ in \eqref{def:rf-initial-population} and the
well-known properties of Legendre transforms (see e.g. Lemma 4.5.8
in \cite{DemboZeitouni}; see also Lemma 2.2.5(a) and Exercise
2.2.22 in \cite{DemboZeitouni} for the convexity and the lower
semicontinuity of $\gamma\mapsto\log g(e^\gamma)$).
\end{proof}

We have an immediate consequence of this proposition that
concerns the case with a deterministic initial population equal to
$\mu_g$ (almost surely). Namely, if we consider the probability
generating function $g_\diamond$ defined by
$g_\diamond(s):=s^{\mu_g}$ (for all $s$), then we mean the case
$g=g_\diamond$, and therefore:
\begin{itemize}
\item$V_0^{f,g_\diamond}=\mu_g$ almost surely; thus, $Z_n=\mu_g$
and $\bar{Z}_n=\mu_g$ almost surely (for all $n\geq1$);
\item$\{Y_n^{f,g_\diamond}:n\geq1\}$ are i.i.d. random variables
distributed as $Y^{f,g_\diamond}$, that is,
\[
Y^{f,g_\diamond}:=\mu_g+\sum_{i=1}^\tau
V_i^{f,g_\diamond},\quad \mbox {where}\ \tau:=\inf \bigl\{n
\geq0:V_n^{f,g_\diamond}=0 \bigr\};
\]
\item the rate function
$J_{\mathcal{G}_{f,g_\diamond},g_\diamond}$ is
\begin{equation}
\label{eq:main-estimators-rf-deterministic-initial-population} J_{\mathcal{G}_{f,g_\diamond},g_\diamond}(x)= \left\{ %
\begin{array}{@{}ll}
\mu_g\cdot\frac{I_f(x)}{1-x}&\ \mbox{if}\ x\in[0,1),\\
\infty&\ \mbox{otherwise,}
\end{array} %
\right.
\end{equation}
by Proposition \ref{prop:main-estimators}.
\end{itemize}

\begin{cor}[Comparison between $J_{\mathcal{G}_{f,g},g}$ in Proposition
\ref{prop:main-estimators} and $J_{\mathcal{G}_{f,g_\diamond},g_\diamond
}$]\label{cor:comparison}
We have $J_{\mathcal{G}_{f,g},g}(x)\leq
J_{\mathcal{G}_{f,g_\diamond},g_\diamond}(x)$ for all
$x\in\mathbb{R}$. Moreover the inequality turns into an equality
if and only if we have one of the following cases:
\begin{itemize}
\item$x\notin[0,1)$ and $J_{\mathcal{G}_{f,g},g}(x)=J_{\mathcal
{G}_{f,g_\diamond},g_\diamond}(x)=\infty$;
\item$x=\mu_f$ and $J_{\mathcal{G}_{f,g},g}(x)=J_{\mathcal
{G}_{f,g_\diamond},g_\diamond}(x)=0$;
\item$V_0^{f,g}$ is deterministic, equal to $\mu_g$, and
$J_{\mathcal{G}_{f,g},g}(x)=J_{\mathcal{G}_{f,g_\diamond},g_\diamond}(x)$
for all $x\in\mathbb{R}$.
\end{itemize}
\end{cor}
\begin{proof}
The case $x\notin[0,1)$ is trivial. On the contrary, if $x\in
[0,1)$, then by Jensen's inequality we have
\[
-\log g \bigl(e^{-\frac{I_f(x)}{1-x}} \bigr)=-\log\mathbb{E} \bigl[e^{-\frac{I_f(x)}{1-x}\cdot V_0^{f,g}}
\bigr]\leq\mu_g\cdot\frac{I_f(x)}{1-x};
\]
moreover, the cases where the inequality turns into an equality
follow from the well-known properties of Jensen's inequality.
\end{proof}

\begin{remark}[Comparison between convergence of estimators of $\mu
_f$]\label{rem:comparison}
Assume that $\mu_f>0$ and the initial population is not
deterministic. Then there exists $\eta>0$ such that
\begin{equation}
\label{eq:local-strict-inequality-between-rf} 0<J_{\mathcal{G}_{f,g},g}(x)<J_{\mathcal{G}_{f,g_\diamond},g_\diamond
}(x)\quad \mbox{for}\ x\in(
\mu_f-\eta,\mu_f+\eta)\setminus\{\mu_f\}.
\end{equation}
Thus, we can say that
$ \{\frac{\bar{Y}_n^{f,g_\diamond}-\mu_g}{\bar{Y}_n^{f,g_\diamond
}}:n\geq
1 \}$ converges to $\mu_f$ \textup{(}as $n\to\infty)$ faster than
$ \{\frac{\bar{Y}_n^{f,g}-\bar{Z}_n}{\bar{Y}_n^{f,g}}:n\geq
1 \}$; in fact, we can find $\varepsilon>0$ such that
\[
\lim_{n\to\infty}\frac{P (\llvert \frac{\bar{Y}_n^{f,g_\diamond}-\mu
_g}{\bar{Y}_n^{f,g_\diamond}}-\mu_f\rrvert \geq\varepsilon )}{
P (\llvert \frac{\bar{Y}_n^{f,g}-\bar{Z}_n}{\bar{Y}_n^{f,g}}-\mu
_f\rrvert \geq\varepsilon )}=0.
\]

We can repeat the same argument to say that
$ \{\frac{\bar{Y}_n^{f,g_\diamond}-\mu_g}{\bar{Y}_n^{f,g_\diamond
}}:n\geq
1 \}$ converges to $\mu_f$ \textup{(}as $n\to\infty)$ faster than
$\{\bar{X}_n:n\geq1\}$ in Lemma \ref{lem:LDP-offspring}. In fact,
we have $V_0^{f,g_\diamond}=\mu_g$ almost surely, $\mu_g$ is an
integer, and, since $\mu_g>0$ because $q_0=0$, we have $\mu_g\geq
1$; then we have
\[
J_{\mathcal{G}_{f,g_\diamond},g_\diamond}(x)=\mu_g\cdot\frac
{I_f(x)}{1-x}>I_f(x)>0
\quad \mbox{for all}\ x\in(0,1)\setminus\{\mu_f\}
\]
\textup{(}we can also consider the case $x=0$ if $\mu_g>1)$.
\end{remark}

Now we present the LDP for the second sequence of estimators.

\begin{proposition}\label{prop:minor-estimators}
Assume the same hypotheses of Proposition \ref{prop:main} and
$q_0=0$. Let $\{Y_n:n\geq1\}$ be i.i.d. random variables
distributed as $Y^{f,g}$. Let $\{\bar{Y}_n:n\geq1\}$ be the
sequence of empirical means defined by
$\bar{Y}_n:=\frac{1}{n}\sum_{k=1}^nY_k$ \textup{(}for all $n\geq1)$. Then
$ \{\frac{\bar{Y}_n-\mu_g}{\bar{Y}_n}:n\geq1 \}$
satisfies the LDP with good rate function $J_{\mu_g}$ defined by
\[
J_{\mu_g}(x):= \left\{ %
\begin{array}{@{}ll}
\inf \{\frac{\mu_g}{1-x}I_f (\frac{\frac{\mu_g}{1-x}-z}{\frac
{\mu_g}{1-x}} )+I_g(z):z>0 \}&\ \mbox{if}\ x<1,\\
\infty&\ \mbox{if}\ x\geq1.
\end{array} %
\right.
\]
\end{proposition}
\begin{proof}
By Proposition \ref{prop:main} and the contraction principle we
have the LDP of $ \{\frac{\bar{Y}_n-\mu_g}{\bar{Y}_n}:n\geq
1 \}$ with good rate function $J_{\mu_g}$ defined by
\[
J_{\mu_g}(x):=\inf \biggl\{I_{\mathcal{G}_{f,g},g}(y,z):y\geq z>0,
\frac
{y-\mu_g}{y}=x \biggr\}.
\]
The case $x\geq1$ is trivial because we have the infimum over the
empty set (we recall that $\mu_g>0$ because $q_0=0$). For $x<1$, we
have
\[
J_{\mu_g}(x)=\inf \biggl\{I_{\mathcal{G}_{f,g},g} \biggl(\frac{\mu
_g}{1-x},z
\biggr):z>0 \biggr\},
\]
and we obtain the desired formula by taking into account the
expression of the rate function $I_{\mathcal{G}_{f,g},g}$ in
Proposition \ref{prop:main}.\vadjust{\eject}
\end{proof}

\begin{remark}[We can have $J_{\mu_g}(x)<\infty$ for some $x<0$]\label
{rem:rf-in-propminorestimator-could-be-finite-for-negative-arguments}
We know that, for $J_{\mathcal{G}_{f,g},g}$ in Proposition
\ref{prop:main-estimators}, we have
$J_{\mathcal{G}_{f,g},g}(x)=\infty$ for $x\notin[0,1)$. On the
contrary, as we see, we could have $J_{\mu_g}(x)<\infty$ for some
$x<0$. In order to explain this fact, we denote the minimum value
$r$ such that $q_r>0$ by $r_{\mathrm{min}}$; then we have
$\mu_g\geq r_{\mathrm{min}}$; moreover, we have
$\mu_g>r_{\mathrm{min}}$ if $q_{r_{\mathrm{min}}}<1$. In
conclusion, we can say that if $\mu_g>r_{\mathrm{min}}$, then the range
of negative values $x$ such that $J_{\mu_g}(x)<\infty$ is
\begin{equation}
\label{eq:range-of-negative-x} x\geq1-\frac{\mu_g}{r_{\mathrm{min}}};
\end{equation}
in fact, for $x<1$, both
$I_f (\frac{\frac{\mu_g}{1-x}-z}{\frac{\mu_g}{1-x}} )$
and $I_g(z)$ are finite for $z\in
[r_{\mathrm{min}},\frac{\mu_g}{1-x}]$, and therefore we can say
that $J_{\mu_g}(x)<\infty$ if
$r_{\mathrm{min}}\leq\frac{\mu_g}{1-x}$ or, equivalently, if
\eqref{eq:range-of-negative-x} holds.
\end{remark}

\begin{remark}[Estimators of $\mu_f$ when $\mu_f=0$]\label{rem:no-offsprings}
If $\mu_f=0$, that is, $f(s)=1$ for all $s$ or, equivalently, $p_0=1$,
then the rate function in Proposition \ref{prop:main-estimators} is
\[
J_{\mathcal{G}_{f,g},g}(x)= \left\{ %
\begin{array}{@{}ll}
0&\ \mbox{if}\ x=0,\\
\infty&\ \mbox{otherwise}.
\end{array} %
\right.
\]
Then it is easy to check that $J_{\mathcal{G}_{f,g},g}$ coincides
with $I_f$, and therefore $J_{\mathcal{G}_{f,g},g}$ coincides with
$J_{\mathcal{G}_{f,g_\diamond},g_\diamond}$ in
\eqref{eq:main-estimators-rf-deterministic-initial-population}
\textup{(}note that, in particular, we cannot have the strict
inequalities in
\eqref{eq:local-strict-inequality-between-rf} in Remark
\ref{rem:comparison} stated for the case $\mu_f>0$). Finally, if
$\mu_f=0$ \textup{(}and as usual $q_0=0$ or, equivalently, $\mu_g>0$),
then we
have $z=\frac{\mu_g}{1-x}$ in the variational formula of the rate
function in Proposition \ref{prop:minor-estimators}, and therefore
\begin{equation}
\label{eq:rf-prop-minor-estimators-muf=0} J_{\mu_g}(x)= \left\{ %
\begin{array}{@{}ll}
I_g (\frac{\mu_g}{1-x} )&\ \mbox{if}\ 1-\frac{\mu_g}{r_{\mathrm
{min}}}\leq x<1,\\
\infty&\ \mbox{otherwise}.
\end{array} %
\right.
\end{equation}
Note the rate function in
\eqref{eq:rf-prop-minor-estimators-muf=0} can also be derived by
combining the contraction principle and the rate function $I_g$
for the empirical means $\{\bar{Z}_n:n\geq1\}$; in fact, we have
$ \{\frac{\bar{Y}_n-\mu_g}{\bar{Y}_n}:n\geq
1 \}= \{\frac{\bar{Z}_n-\mu_g}{\bar{Z}_n}:n\geq
1 \}$, and the rate function $I_g$ is good by the hypotheses
of Proposition \ref{prop:minor-estimators} \textup{(}see Proposition
\ref{prop:main} and Remark
\ref{rem:implicit-hypotheses-for-prop-main}\textup{)}. Finally, we also note
that inequality \eqref{eq:range-of-negative-x} appears in the
rate function expression~\eqref{eq:rf-prop-minor-estimators-muf=0}.
\end{remark}

\section*{Acknowledgments}
The authors thank a referee for suggesting shorter proofs of
Propositions \ref{prop:main-unitary-initial-population} and
\ref{prop:main}. The support of GNAMPA (INDAM) is acknowledged.


%
\end{document}